\documentclass[12pt]{amsart}
\usepackage{amssymb, amsmath, amsfonts, amsthm, esint,graphicx}
\usepackage[dvipsnames]{xcolor}
\usepackage[alphabetic]{amsrefs}
\usepackage[margin=0.8 in]{geometry}
\usepackage[colorlinks=true, linkcolor=blue, citecolor=Plum]{hyperref}
\definecolor{newcolor}{rgb}{0.1,1,0}
        
        \newcommand{\be}{\begin{equation}}
        \newcommand{\ee}{\end{equation}}
        \newcommand{\ba}{\begin{eqnarray}}
        \newcommand{\ea}{\end{eqnarray}}
        \newcommand{\ban}{\begin{eqnarray*}}
        	\newcommand{\ean}{\end{eqnarray*}}
        
        \def \dd {\partial}

\def \dd {\partial}

\def \eps {\varepsilon}

\DeclareMathOperator{\Ric}{Ric}

\DeclareMathOperator{\Hess}{Hess}
\DeclareMathOperator{\diam}{diam}

\DeclareMathOperator{\sn}{sn}
\DeclareMathOperator{\tn}{tn}
\DeclareMathOperator{\cs}{cs}

\DeclareMathOperator{\SecFun}{II}

\title{Fundamental Gap Estimate for Convex Domains on Sphere --- the case $n=2$}
\author{Xianzhe Dai}
\address[Xianzhe Dai]{Department of Mathematics, ECNU, Shanghai and UCSB, Santa Barbara CA 93106 }
\email{\href{mailto:shoseto@ucsb.edu}{dai@math.ucsb.edu}}
\thanks{Partially supported by NSF DMS and NSFC}
\author{Shoo Seto}
\address[Shoo Seto]{Department of Mathematics\\
	University of California\\
	Santa Barbara, CA 93106}
\email{\href{mailto:shoseto@ucsb.edu}{shoseto@ucsb.edu}}
\thanks{Partially supported by Simons Foundation}
\author{Guofang Wei}
\address[Guofang Wei]{Department of Mathematics\\
	University of California\\
	Santa Barbara, CA 93106}
\email{\href{mailto:wei@math.ucsb.edu}{wei@math.ucsb.edu}}
\thanks{Partially supported by NSF DMS 1506393}
\date{}

\begin{document}
	
\theoremstyle{definition} 
\newtheorem{theorem}{Theorem}[section]
\newtheorem{definition}[theorem]{Definition}
\newtheorem{conjecture}{Conjecture}[section]
\newtheorem{example}{Example}[section]
\newtheorem{lemma}[theorem]{Lemma}
\newtheorem{remark}[theorem]{Remark}
\newtheorem{question}{Question}[section]
\newtheorem{proposition}[theorem]{Proposition}
\newtheorem{corollary}[theorem]{Corollary}
\newtheorem*{notation}{Notation}

\numberwithin{equation}{section}
	\begin{abstract}
		In \cite{seto-wang-wei, he-wei} it is shown that the difference of the first two eigenvalues of the Laplacian with Dirichlet boundary condition on convex domain with diameter $D$ of sphere $\mathbb S^n$  is $\ge 3 \frac{\pi^2}{D^2}$  when $n \ge 3$. We prove the same result when $n=2$. In fact our proof works for all dimension.  We also give an asymptotic expansion of the first and second Dirichlet eigenvalues of the model in \cite{seto-wang-wei}.
	\end{abstract}
\maketitle
\section{Introduction}
Let $M$ be an $n$-dimensional Riemannian manifold and $\Omega \subset M$ a bounded convex domain with diameter $D$.  The spectrum of the Laplacian on $\Omega$ with respect to the Dirichlet or the Neumann boundary condition is nonnegative and discrete.  Furthermore, the first Dirichlet eigenvalue, $\lambda_1$, is positive and simple so that we can define the \textit{fundamental gap} as
\begin{equation*}
\Gamma(\Omega) := \lambda_2 - \lambda_1 >0.
\end{equation*}
There is a rich history towards estimating a lower bound for the fundamental gap depending on geometric data.  In particular, for convex domains in $\mathbb R^n$, the fundamental gap conjecture states that the fundamental gap is  $\ge \frac{3\pi^2}{D^2}$, where $D$ is the diameter of the convex domain.   This  was proven by B. Andrews and J. Clutterbuck in their celebrated work \cite{andrews-clutterbuck}.  When $M=\mathbb S^n$, \cite{seto-wang-wei} proved the same lower bound for dimensions $n\geq 3$ and diameter $D < \frac{\pi}{2}$.  The diameter restriction was removed by C. He and the third author in \cite{he-wei} by using parabolic methods and a delicate construction of supersolutions to a  one-dimensional nonlinear parabolic model.  In fact, in the work of \cite{seto-wang-wei}, the estimate holds for $\mathbb{M}^n_K$, the simply connected spaces with constant curvature $K$, with $K\geq 0$.  In this paper, by using a different model,  we show that the fundamental gap estimate for convex domain in $\mathbb S^n$ also holds for $n=2$. In fact the proof works for all $n$ and $K\ge 0$. 
	\begin{theorem} \label{main1}
	Let $\Omega \subset \mathbb M_K^n (K \ge 0)$ be a strictly convex domain with diameter $D$, $\lambda_i$ $(i=1,2)$ be the first two eigenvalues of the Laplacian on $\Omega$ with Dirichlet boundary condition. Then
	\begin{equation} \label{gap-est}
		\lambda_2-\lambda_1 \ge 3 \frac{\pi^2}{D^2}.
	\end{equation}
\end{theorem}
The key to proving this is to show the following log-concavity of the first eigenfunction.
\begin{theorem}\label{log-concavity}
		Given $\Omega \subset \mathbb M^n_K$ a bounded strictly convex domain with diameter $D$ and $K\geq 0$, let
	$\phi_1>0$ be a first eigenfunction of the Laplacian on $\Omega$.  Then $\forall x, y  \in \Omega$, with $x \not= y$, and $\gamma(t), \  t\in [ - \tfrac{d}{2}, \tfrac{d}{2}]$ the unique unit-speed length minimizing geodesic connecting $x$ to $y$, 
	\be
	\langle \nabla \log \phi_1 (y), \gamma'(\tfrac{d}{2}) \rangle - \langle \nabla\log \phi_1(x),\gamma'(-\tfrac{d}{2})\rangle  \leq -2 \tfrac{\pi}{D} \tan \left( \tfrac{\pi d}{2D} \right)  + (n-1)\tn_K (\tfrac{d}{2})		
	\ee
	holds (see (\ref{tn_K}) for the definition of $\tn_K$), which gives
	\[
	\Hess \,(\log \phi_1) \le - \left(\frac{\pi^2}{D^2} - \frac{n-1}{2} K\right) \,\mbox{id}.\]
\end{theorem}
When $K=0$, this recovers the log-concavity proved in \cite{andrews-clutterbuck}. When $n=3$, this log-concavity is the same log-concavity as in \cite[Theorem 1.5]{seto-wang-wei}, referred as sphere model. In general there is no direct comparison.  But when $KD^2$ is small,  this log-concavity is worse than the sphere model for $n> 3$ but better than the sphere model for $n = 2$, see Remark~\ref{two-mod}  for details.	

We also give an asymptotic expansion of the first and second Dirichlet eigenvalues of the sphere model in \cite{seto-wang-wei}. Recall $\bar{\lambda}_1 (n,D,K), \ \bar{\lambda}_2 (n,D,K)$ are the first and second Dirichlet eigenvalues of
\begin{equation}\label{schrodingernormal-1}
\varphi''(s) - \tfrac{(n-1)K}{4} \left( \tfrac{n-3}{\cs_K^2(s)}  - (n-1) \right)  \varphi=  - \lambda \, \varphi
\end{equation}
on $[-\tfrac{D}{2},  \tfrac{D}{2}]$ (see (\ref{sn_k-cs_K}) for definition of $\cs_K$).
  When $n=1,3$ or $K=0$, one can find the eigenvalues and eigenfunctions explicitly and the gap  $ \bar{\lambda}_2 (n,D,K) -\bar{\lambda}_1 (n,D,K) = 3 \frac{ \pi^2}{D^2}$.  In general one can not find the eigenvalues explicitly. 
When $K>0$,  as $(\cs_K^{-2}(s))'' \ge 0$, $ \bar{\lambda}_2 (n,D,K) -\bar{\lambda}_1 (n,D,K) > 3 \frac{ \pi^2}{D^2}$ when  $n > 3$,  but $<3 \frac{ \pi^2}{D^2}$ when  $n =2$ \cite{ashbaughbenguria89}.

\begin{proposition} \label{eigen-asym}
For $K\in \mathbb{R}$,
\begin{equation*}
\bar\lambda_1 = \frac{\pi^2}{D^2}-\frac{(n-1)}{2}K+\frac{(n-1)(n-3)}{48\pi^2}(\pi^2 - 6)D^2K^2+\frac{(n-1)(n-3)}{480\pi^4}D^4K^3(\pi^4-20\pi^2+120) +O(K^4).
\end{equation*}
and
\begin{equation*}
\bar\lambda_2 = \frac{4\pi^2}{D^2} -\frac{(n-1)}{2}K+\frac{(n-1)(n-3)}{48\pi^2}\left(\pi^2-\frac{3}{2}\right)D^2K^2+\frac{(n-1)(n-3)}{480\pi^4} D^4K^3\left(\pi^4-5\pi^2+\frac{15}{2}\right) +O(K^4).
\end{equation*}
Hence 
\[
	\bar{\lambda}_2 (n,D,K) -\bar{\lambda}_1 (n,D,K)  =  3 \frac{ \pi^2}{D^2} + \frac{3(n-1)(n-3)}{32}  \frac{D^2K^2}{\pi^2} + \frac{(n-1)(n-3)}{480\pi^4}D^4K^3\left(15\pi^2-\frac{225}{2}\right) +O(K^4)
\]
and  for $n \ge 3$, $K$ small,\begin{equation}
\bar{\lambda}_2 (n,D,K) -\bar{\lambda}_1 (n,D,K)  \ge  3 \frac{ \pi^2}{D^2} + \frac{3(n-1)(n-3)}{32}  \frac{D^2K^2}{\pi^2}. \label{gap-asym}	\end{equation}
\end{proposition}

\begin{remark}
The estimate (\ref{gap-asym}) gives an explicit lower bound which is bigger than $ 3 \frac{ \pi^2}{D^2}$ when $KD^2$ is small and $n \ge 3$. On the other hand the estimate seems to be not true when $KD^2$ is big. In fact beginning with the $K^5$ order, the coefficient changes sign for some $n >3$, instead of at $n=3$, see Section~\ref{high-order}.
\end{remark}

\subsection*{Outline of the paper}
In \S 2 we establish the notations, definitions and preliminary lemmas which we will use.  In \S 3, we prove the key result on the log-concavity of the first eigenfunction by comparing with the one-dimensional model.   In \S 4, we apply the log-concavity result to compare the gap of the first and second eigenvalues between convex domains of spheres and the one-dimensional model.  In \S 5, we compute the asymptotics of the first and second eigenvalues of the one-dimensional model used in \cite{seto-wang-wei}.  The analysis of the one-dimensional model in \S 5 is interesting on its own and can be read independently.  

\subsection*{Acknowledgement} We would like to thank Chenxu He for very careful reading of the first version and very  helpful comments and conversations. 

\section{Preliminaries}
We use the following notation
\begin{equation}  \label{sn_k-cs_K}
\sn_K(s) =
\begin{cases}
\frac{1}{\sqrt{K}}\sin(\sqrt{K}s), & K > 0 \\
s, & K=0\\
\frac{1}{\sqrt{-K}}\sinh(\sqrt{-K}s) & K < 0,
\end{cases}
\quad \text{ and }\quad
\cs_K(s) =
\begin{cases}
\cos(\sqrt{K}s), & K > 0 \\
1, & K=0\\
\cosh(\sqrt{-K}s), & K<0,
\end{cases}
\end{equation}
and
\begin{equation}
\tn_K(s) =
\begin{cases}
\sqrt{K}\tan(\sqrt{K}s), & K > 0 \\
0, & K=0 \\
-\sqrt{-K}\tanh(\sqrt{-K}s) & K <0.
\end{cases}  \label{tn_K}
\end{equation}
\begin{definition}
Given a semi-convex function $u$ on a domain $\Omega$, a function $\psi:[0,+\infty) \to \mathbb{R}$ is called a \textit{modified modulus of concavity} for $u$ if for every $x\neq y$ in $\Omega$,
\begin{equation*}
\langle \nabla u(y),\gamma'(\tfrac{d}{2})\rangle - \langle \nabla u(x),\gamma'(-\tfrac{d}{2})\rangle \leq 2 \psi(\tfrac{d}{2}) + (n-1)\tn_K(\tfrac{d}{2}),
\end{equation*}
where $\gamma$ is the unit-speed length minimizing geodesic with $\gamma(-\tfrac{d}{2}) = x$ and $\gamma(\tfrac{d}{2}) = y$, $d=d(x,y)$.
\end{definition}
The main tool we will use is the following preservation of the modified modulus of concavity under the one-dimensional flow.
\begin{theorem}[Theorem 3.6 \cite{seto-wang-wei}] \label{log-con-preserve2}
Let $\Omega\subset \mathbb{M}_K^n$ be a uniformly convex domain with diameter $D$, where $K\geq 0$.  Let $\phi_1$ be a positive first eigenfunction of the Laplacian on $\Omega$ with Dirichlet boundary condition associated to the eigenvalue $\lambda_1$, and $u:\Omega \times \mathbb{R}_+\to \mathbb{R}$ be given by $u(x,t) = e^{-\lambda_1t}\phi_1(x)$. Suppose $\psi_0:[0,D/2] \to \mathbb{R}$ satisfies
\begin{equation*}
\langle \nabla \log u(y,0), \gamma'(\tfrac{d}{2}) \rangle - \langle \nabla\log u(x,0),\gamma'(-\tfrac{d}{2})\rangle \leq 2\psi_0|_{s=\frac{d}{2}} + (n-1)\tn_K(\tfrac{d}{2}).
\end{equation*}
Let $\psi \in C^0([0,D/2])\times \mathbb{R}_+) \cap C^\infty([0,D/2]\times (0,\infty))$ be a solution of
\begin{equation}\label{psiflow}
\begin{cases}
\frac{\dd\psi}{\dd t} \geq \psi''(s,t) + 2\psi\psi'(s,t) -2\tn_K(s)(\psi'(s,t) + \psi^2(s,t) + \lambda_1) & \text{ on } [0,D/2]\times \mathbb{R}_+ \\
\psi(\cdot,0) = \psi_0(\cdot) \\
\psi(0,t) = 0 \\
\psi(s,t) \leq 0.
\end{cases}
\end{equation}
Then
\begin{equation*}
\langle \nabla \log u(y,t), \gamma'(\tfrac{d}{2}) \rangle - \langle \nabla\log u(x,t),\gamma'(-\tfrac{d}{2})\rangle \leq 2\psi(s,t)|_{s=\frac{d}{2}} + (n-1)\tn_K(\tfrac{d}{2})
\end{equation*}
for all $t \geq 0$ and $D \leq \frac{\pi}{\sqrt{K}}$ if $K>0$.
\end{theorem}
\begin{remark}\label{remark}
Note that the stationary solutions of $\psi$ satisfy
\begin{equation*}
0= (\psi'(s) + \psi^2(s) +\lambda_1)' - 2\tn_K(s)(\psi'+\psi^2(s)+\lambda_1).
\end{equation*}
Solving the ODE $y'-2\tn_K(s)y = 0$, we have $y= y(0)\cs_K^{-2}(s)$.  Hence an initial condition $y(0) = 0$ would imply the trivial solution in $y$, which is equivalent to $\psi'+\psi^2 + \lambda_1=0$.  The condition $y(0)=0$ can be obtained by adding the condition $\psi'(0) = -\lambda_1$. 
\end{remark}

Additionally, we will use the following two lemmas which control the Hessian log of positive functions vanishing at the boundary.  Note that the function is not necessarily the first eigenfunction.  We first look at the Hessian log itself near the boundary and in the interior.
\begin{lemma}[Lemma 3.4 \cite{seto-wang-wei}, Lemma 4.2 \cite{andrews-clutterbuck}]\label{Hessianest}
Let $\Omega$ be a uniformly convex bounded domain in a Riemannian manifold $M^n$,
and $u:  \overline{ \Omega}  \times \mathbb{R}_+ \to \mathbb{R}$ a $C^2$ function such that $u$ is positive on $\Omega$, $u(\cdot, t) = 0$ and  $\nabla u  \neq 0$ on $\dd\Omega$.   Given $T<\infty$,  there exists $r_1> 0$ such that $\nabla^2\log u|_{(x,t)} < 0$ whenever $d(x,\dd\Omega) < r_1$ and $t\in [0,T]$, and $N \in \mathbb{R}$ such that $\nabla^2\log u|_{(x,t)}(v,v) \leq N\|v\|^2$ for all $x \in \Omega$ and $t \in [0,T]$.
\end{lemma}
The next lemma controls the modulus of log concavity near the boundary.  Let $\hat\Omega := \Omega \times \Omega -\{(x,x)\ | \ x \in \Omega\}$.
\begin{lemma}[Lemma 3.5 \cite{seto-wang-wei}, Lemma 4.3 \cite{andrews-clutterbuck}] \label{boundary-cover-lemma}
Let $\Omega$ and $u$ be as in Lemma \ref{Hessianest} and let $\psi$ be continuous on $[0,D/2]\times\mathbb{R}_+$ and Lipschitz in the first argument, with $\psi(0,t) = 0$ for each $t$ with $D = $ diam $\Omega$. Then for any $T < \infty$ and $\beta > 0$, there exists an open set $U_{\beta,T} \subset M\times M$ containing $\dd\hat{\Omega}$ such that 
\begin{equation*}
\langle \nabla \log u(y,t),\gamma'(\tfrac{d}{2})\rangle - \langle \nabla \log u(x,t),\gamma'(-\tfrac{d}{2})\rangle - 2\psi\left(\frac{d(x,y)}{2},t\right) < \beta,
\end{equation*}
for all $t\in [0,T]$ and $(x,y) \in U_{\beta,T}\cap \hat{\Omega}$.
\end{lemma}

In order to use Theorem \ref{log-con-preserve2}, we need to show that our model satisfies the differential inequality.  
\begin{lemma}\label{eigvalcomp}
Let $\lambda_1$ be the first eigenvalue of the Laplacian on a convex domain $\Omega \subset S^n$ with $\diam\Omega =D$.  Then
\begin{equation}
\frac{\pi^2}{D^2} \leq \lambda_1.  \label{lambda1-lb}
\end{equation}
\end{lemma}
\begin{remark}
This can be shown by comparing the Neumann eigenvalues, indexed by $0 = \mu_0 < \mu_1 \leq \ldots $, and Dirichlet eigenvalues on the sphere, namely for domains $\Omega \subset S^n$ whose boundary has nonnegative mean curvature 
\begin{equation*}
\mu_k(\Omega) \leq \lambda_k(\Omega), \quad \forall k \geq 1.
\end{equation*}
This result can be found in \cite{ashbaugh-levine} or \cite{hsu-wang}.  Since $\frac{\pi^2}{D^2} \leq \mu_1(\Omega)$, where $D=\diam(\Omega)$, one has (\ref{lambda1-lb}).  We present an alternative short argument.
\end{remark}
\begin{proof}
By domain monotonicity for Dirichlet eigenvalues, it suffices to show the lower bound for balls since they are maximally convex sets.  By separation of variables, the first eigenfunction is given by
\begin{equation*}
-y''-(n-1)\cot(x)y' = \lambda_1y,  \quad \text{ on } (0,\tfrac{D}{2}).
\end{equation*}
with $y'(0) = 0$, $y(\tfrac{D}{2}) =0$ and normalized so that $y(0) =1$. From the Rayleigh quotient on Euclidean space, we have
\begin{align*}
\frac{\pi^2}{D^2} &\leq \frac{\int_0^{\frac{D}{2}}(y')^2}{\int_0^{\frac{D}{2}}y^2} 
= -\frac{\int_0^{\frac{D}{2}}yy''}{\int_0^{\frac{D}{2}}y^2}\\
&= \frac{(n-1)\int_0^{\frac{D}{2}}\cot(x)yy'+\lambda_1\int_0^{\frac{D}{2}}y^2}{\int_0^{\frac{D}{2}}y^2} \\
&\leq \lambda_1,
\end{align*}
since $\cot(x) \geq 0$, $y \geq 0$ and $y'\leq 0$. (c.f. \cite{ashbaugh-benguria}).
\end{proof}

\section{Proof of Theorem \ref{log-concavity}}
To prove the log-concavity estimate, we first need to construct a suitable initial function $\psi_0$ and then improve it by flowing so that it limits to the model solution. The construction is motivated and parallel to the one in \cite{andrews-clutterbuck, he-wei}. 

 Note that for $\phi_{0} = \cos(\tfrac{\pi s}{D})$, $\psi_0=(\log \phi_{0})'$ is a stationary solution of (\ref{psiflow}) that comes with the trivial solution for the ODE in Remark \ref{remark} which will satisfy the differential inequality.  Thus we are interested in solutions of the ODE
\begin{equation}\label{stationarysolution}
\psi'(s) + \psi^2(s) + \frac{\pi^2}{D^2} = -\frac{c}{\cs_K^2(s)},
\end{equation}
where $c$ is some constant (Note the difference in sign convention in \cite{he-wei}).  We choose the value $\frac{\pi^2}{D^2}$ here so that the solution will converge to the Euclidean model.
Note also that one of the boundary conditions for $\psi_0$ is singular. Therefore, we approximate it by a monotone sequence whose boundary values are regular. To this end, fix an integer $k >0$ and consider the solutions $\psi_c^L$ and $\psi^R_{c,k}$ with
\begin{equation}\label{lefteqn}
\begin{cases}
(\psi_c^L)' + (\psi^L_c)^2 + \frac{\pi^2}{D^2}+\frac{c}{\cs_K^2(s)}=0, & \text{on }(0,D/2)\\
\psi_c^L(0) = 0,
\end{cases}
\end{equation}
and
\begin{equation}\label{righteqn}
\begin{cases}
(\psi_{c,k}^R)' + (\psi^R_{c,k})^2 + \frac{\pi^2}{D^2}+\frac{c}{\cs_K^2(s)}=0, & \text{on }(0,D/2)\\
\psi_{c,k}^R(\tfrac{D}{2}) = -k.
\end{cases}
\end{equation}
In the following, we will first note that the solutions can be constructed by turning the Riccati equation into a second order linear equation and then solving it via the Pr\"ufer transformation. Then we point out that, for specific $c=c_k$ and $k$ sufficiently large,  the solutions comes from a Robin eigenvalue problem (with additional normalization).

Indeed, consider the second order linear equation
\begin{equation}\label{2ndordereqn}
\phi''(s)+\frac{\pi^2}{D^2}\phi(s)=-\frac{c}{\cs_K^2(s)}\phi(s), \   \text{ on }[0,D/2]
\end{equation}
The solutions to (\ref{2ndordereqn}) and the solutions to (\ref{stationarysolution}) are related by $\psi=(\log \phi)'$. Therefore we need positive solutions for (\ref{2ndordereqn}).

The Pr\"ufer transformation construction of the solution to (\ref{2ndordereqn}) is to consider a ``polar coordinate'' of the solutions
\begin{equation*}
\begin{cases}
\phi'(z) = r(z)\sin(q(z))\\
\phi(z) = r(z)\cos(q(z)),
\end{cases}
\end{equation*}
for some function $r(z)$ and $q(z)$.
The functions $q(z) = \arctan\left( \frac{\phi'(z)}{\phi(z)} \right)$ and $r^2(z) = (\phi'(z))^2+\phi^2(z)$ satisfies a system of first order ODEs
\begin{equation}\label{qsystem}
\begin{cases}
\frac{d q}{dz} =-\left(\frac{c}{\cs_K^2(z)}+\frac{\pi^2}{D^2}\right)\cos^2(q)-\sin^2(q)\\
q(0,q_0,c)=q_0,
\end{cases}
\end{equation}
and
\begin{equation}\label{rsystem}
\begin{cases}
\frac{dr}{dz} = \left(1-\frac{c}{\cs_K^2(z)}-\frac{\pi^2}{D^2}\right)r(z)\cos(q(z))\sin(q(z))
\end{cases}
\end{equation}
The system is partially decoupled and we solve \eqref{qsystem} first and then \eqref{rsystem}.  Then 
\begin{equation*}
\phi^L_{c}(z) = \exp\left( \int_0^{z}\tan q(s,0,c)ds\right)
\end{equation*}
is the positive solution which corresponds to the solution to (\ref{lefteqn}). Similarly
\begin{equation*}
\phi^R_{c, k}(z) = \exp\left( \int_{z}^{\frac{D}{2}}\tan q(\frac{D}{2}-s,k,c)ds\right)
\end{equation*}
gives rise to the solution of (\ref{righteqn}).

We now observe that, for sufficiently large $k$ and specific $c=c_k$, both solutions coincide and come from an eigenvalue problem. First of all, by ODE comparison, we see that $q(z,q_0,c)$ is strictly decreasing in $c$ for all $z$.  Furthermore, when $q_0 = 0$ and $c = 0$, this corresponds to the model situation $\phi_0=\cos\left(\frac{\pi}{D}z\right)$.  In terms of $q$, we have $q(\tfrac{D}{2},0,0) = -\frac{\pi}{2}$. 

 Therefore, for sufficiently large $k$, there exists a unique $c_k<0$ such that $$q(\tfrac{D}{2},0,c_k) = -\frac{\pi}{2}+\arctan(\tfrac{1}{k}).$$

Then
\begin{equation*}
\phi_{0,1/k}(z) = \frac{1}{k}\exp\left( -\int_{z}^{\frac{D}{2}}\tan q(s,0,c_k)ds\right)
\end{equation*}
with
\begin{align*}
\phi_{0,1/k}'(\tfrac{D}{2}) &= \frac{1}{k}\tan\left(\arctan(k^{-1})-\frac{\pi}{2}\right)\\
&=-\frac{1}{k}\cot(\arctan(k^{-1})) = -1.
\end{align*}
is the solution to the Robin eigenvalue problem (with additional normalization)
\begin{equation}\label{robinboundaryeqn}
\begin{cases}
(\phi_{0,1/k})''(s)+\frac{\pi^2}{D^2}\phi_{0,1/k}(s)=-\frac{c_k}{\cs_K^2(s)}\phi_{0,1/k}(s) & \text{ on }[0,D/2]\\
\phi_{0,1/k}(\tfrac{D}{2})=1/k\\
\phi'_{0,1/k}(\tfrac{D}{2})=-1\\
\phi'_{0,1/k}(0)=0\\
\phi_{0,1/k} > 0 & \text{ on }[0,D/2].
\end{cases}
\end{equation}

With this unique choice of $c_k$, we have $\psi^L_{c_k}=\psi^R_{ k, c_k} = (\log \phi_{0,1/k})'$.  
\begin{remark}
When $k \to \infty$ and $c \to 0$, the solution is given explicitly by $\phi_{0,0} = \phi_{0}=\cos(\tfrac{\pi s}{D})$.
\end{remark}
\begin{remark}
The constant $c_k$ in the Robin eigenvalue problem \eqref{robinboundaryeqn} depends on the value $k$ and is unique; in fact it is the smallest eigenvalue.  Therefore the equality $\psi^L_{c_k}=\psi^R_{k,c_k} = (\log \phi_{0,1/k})'$ holds for the specific choice of $c_k$ when $k$ is fixed.  In the following section, we show how the different choices for $c$ in \eqref{lefteqn} and \eqref{righteqn} affect the solutions.
\end{remark}

\subsection{Construction of supersolution}
Unlike the case of Andrews-Clutterbuck \cite{andrews-clutterbuck}, we do not have freedom in choosing different values for the eigenvalue $\frac{\pi^2}{D^2}$ in \eqref{stationarysolution} to use in our comparison.  However, we have freedom in the choice of $c$.  Using different value for $c$, we will obtain upper and lower bounds of our supersolution. By the ODE comparison, $\psi_c^L$ is strictly decreasing in $c$ on $0 < z \leq \frac{D}{2}$ and $\psi_{k,c}^R$ is strictly increasing in $c$.  Now for $k$, there is some fixed $c_k$ that solves \eqref{stationarysolution} via \eqref{robinboundaryeqn}. So for $c < c_k$ we have $\psi^L_{c} > (\log\phi_{0,1/k})'$ on $0 < z \leq \frac{D}{2}$ and for $c>c_k$ we have $\psi^R_{k,c} > (\log\phi_{0,1/k})'$ on $0\leq z < \frac{D}{2}$.

To obtain upper bounds, for $\lambda_+^2 \geq -c-\frac{\pi^2}{D^2}$, by ODE comparison, we have
\begin{equation*}
\psi_c^L(z) \leq \lambda_+ \tanh(\lambda_+z),
\end{equation*}
and for $\lambda_-^2 \geq \frac{c}{\cs_K^2(\frac{D}{2})}+\frac{\pi^2}{D^2}$, we have
\begin{equation*}
\psi^R_{k,c}(z) \leq \frac{\lambda_-\tan(\lambda_-(\tfrac{D}{2}-z))-k}{1+\tfrac{k}{\lambda_-}\tan(\lambda_-(\tfrac{D}{2}-z))}, \quad z>\frac{D}{2}-\frac{\tfrac{\pi}{2}+\arctan(\tfrac{k}{\lambda_-})}{\lambda_-}.
\end{equation*}

With the upper and lower bound, we can show existence of the supersolution
\begin{equation*}
\psi^+_{k,s} := \min\{\psi_{c_k-s}^L,\psi^R_{k,c_k+s}\}
\end{equation*}
for any $s \geq 0$.  This is a supersolution since both are bounded below by the solution $(\log\phi_{0,1/k})'$ for all $s\geq 0$.

\subsection{Lower bound of supersolution}
Next we show lower bounds of $\psi^+_{k,s}$ for large $s$ so that the supersolution is a modulus of concavity initially.  For
\begin{equation*}
s > \{ c_k+\tfrac{\pi^2}{D^2},-c_k-\tfrac{\pi^2}{D^2} \},
\end{equation*}
let 
\begin{align*}
\tilde\lambda_+ &= \sqrt{s-c_k-\tfrac{\pi^2}{D^2}}\\
\tilde\lambda_- &= \sqrt{s+c_k+\tfrac{\pi^2}{D^2}}.
\end{align*}
Since $\psi^L_{c_k-s}$ solves
\begin{equation*}
\psi'+\psi^2=-\tfrac{\pi^2}{D^2}-\frac{c_k-s}{\cs_K^2(z)} \geq \tilde\lambda^2_+
\end{equation*}
so that by ODE comparison, we have
\begin{equation*}
\psi^L_{c_k-s}(z) \geq \tilde\lambda_+\tanh(\tilde\lambda_+z), \quad 0 \leq z \leq z_0.
\end{equation*}
Similarly, $\psi_{c_k+s}^R$ solves
\begin{equation*}
\psi'+\psi^2=-\tfrac{\pi^2}{D^2}-\frac{c_k+s}{\cs_K^2(z)} \leq - \tilde\lambda_-^2,
\end{equation*}
so that
\begin{equation*}
\psi^R_{c_k+s}(z) \geq \frac{\tilde\lambda_-\tan(\tilde\lambda_-(\tfrac{D}{2}-z))-k}{1+\frac{k}{\tilde\lambda_-}\tan(\tilde\lambda_-(\tfrac{D}{2}-z))},\quad z_- \leq z \leq \frac{D}{2},
\end{equation*}
where $z_0 > \tfrac{D}{2}-\tilde\lambda_-^{-1}(\tfrac{\pi}{2}+\arctan(\tfrac{k}{\tilde\lambda_-}))$.

\subsection{Supersolution is an initial modulus}
Next we show that for each $k$, there is a sufficiently large $s$ such that $\psi_{k,s}^+$ is a modified modulus of concavity for $\log u_0$.

Using Lemma \ref{Hessianest}, there exists $N\in \mathbb{R}$ such that for all $x,y \in \Omega$,
\begin{align*}
\langle \nabla \log u(y,t),\gamma'(\tfrac{d}{2})\rangle - \langle \nabla \log u(x,t),\gamma'(-\tfrac{d}{2})\rangle &\leq \nabla^2\log u (\gamma',\gamma')d(x,y)\\
&\leq Nd(x,y)\\
&\leq 2\lambda \tanh\left(\frac{\lambda d(x,y)}{2}\right),
\end{align*}
where we choose $\lambda$ such that $ND\leq 2\lambda \tanh(\lambda D/2)$.

Next using Lemma \ref{boundary-cover-lemma} with $\psi(z) = \frac{6kz}{D}$ and $\beta = k$, there exists an open set $U\subset M\times M$ containing $\dd\hat\Omega$ ($\hat\Omega :=\Omega \times \Omega - \{(x,x)\ | \ x \in \Omega\}$).  In particular, we can cut out a neighborhood of the diagonal so that there exists a $\delta > 0$ such that $U$ contains all points $x,y \in \Omega$ such that $d(x,y) \geq D- \delta$.  Decreasing so that $\delta < \frac{D}{3}$ if necessary, we have for $d(x,y) \geq D - \delta$ that
\begin{align*}
\langle \nabla \log u(y,t),\gamma'(\tfrac{d}{2})\rangle - \langle \nabla \log u(x,t),\gamma'(-\tfrac{d}{2})\rangle &\leq -\frac{6kd(x,y)}{D}+k \\
&\leq 2\frac{\lambda\tan\left(\lambda\left(\frac{D-d(x,y)}{2}\right)\right)-k}{1+\frac{k}{\lambda}\tan\left(\lambda\left(\frac{D-d(x,y)}{2}\right)\right)},
\end{align*}
for $\lambda > 0$ such that $\frac{D-d(x,y)}{2} < \frac{\frac{\pi}{2}+\arctan(\tfrac{k}{\lambda})}{\lambda}$.  This can be done by choosing $\lambda$ sufficiently large so that $\frac{\pi}{2}+\arctan(\tfrac{k}{\lambda}) < \delta \lambda$. Hence for each $k$, there exists a smallest $s(k) \geq 0$ such that
\begin{equation*}
\langle \nabla \log u_1(y),\gamma'(-\tfrac{d}{2})\rangle - \langle \nabla \log u_1(x),\gamma'(\tfrac{d}{2})\rangle \leq \psi_{k,s(k)}^+\left(\frac{d(x,y)}{2}  \right).
\end{equation*}
Then let
\begin{equation*}
\psi_{k,0}(z) = \min\{\psi_{j,s(j)}^+(z) \ | \ 1 \leq j \leq k\}
\end{equation*}
for $0 \leq z \leq \frac{D}{2}$.  Since $(n-1)\tn_K(s)\geq 0$ for $[0,D/2)$, we can add this term to obtain the initial modified modulus of concavity.

\subsection{Flow into model eigenfunction}
Now we show that given our initial solution we constructed, the following parabolic equation will flow into $\psi = (\log\phi_1)'$.  Then by Theorem \ref{log-con-preserve2}, such a solution will satisfy our required log-concavity condition. Consider
\begin{equation*}
\begin{cases}
\frac{\dd\psi_k}{\dd t} = \psi_k''+2\psi_k\psi_k'-2\tn_K(s)(\psi'_k+\psi_k^2+\frac{\pi^2}{D^2}) & \text{ on }[0,\tfrac{D}{2}]\times\mathbb{R}_+\\
\psi_k(z,0) = \psi_{k,0}(z)\\
\psi_k(0,t) = 0\\
\psi_k(\tfrac{D}{2},t)=-k.
\end{cases}
\end{equation*}
By Lemma \ref{eigvalcomp}, the solution $\psi_k$ satisfies the differential inequality \eqref{psiflow}. Let $u := \psi_k-(\log\phi_{0,1/k})'$ and $f :=(\log\phi_{0,1/k})'$  
Computing, we have
\begin{align*}
2uu' &= 2(\psi_k-f)(\psi_k'-f') = 2\psi_k\psi_k'-2\psi_kf'-2f\psi_k'+2ff'
\end{align*}
and
\begin{align*}
u^2 = \psi_k^2-2\psi_kf+f^2
\end{align*}
and
\begin{align*}
f''+2ff'-2\tn_K(s)(f'+f^2+\tfrac{\pi^2}{D^2})=0.
\end{align*}
By direct computation, we have
\begin{align*}
\frac{\dd u}{\dd t} &= \psi_k''+2\psi_k\psi_k'-2\tn_K(s)(\psi'_k+\psi_k^2+\tfrac{\pi^2}{D^2})\\
&=u''+2uu'-2\tn_K(s)u^2+2u(f'-2\tn_K(s)f)+2(f-\tn_K(s))u'.
\end{align*}
Hence an equivalent equation in $u$ is given by
\begin{equation*}
\begin{cases}
\frac{\dd u}{\dd t} = u'' + 2uu'-2\tn_K(s)u^2+(2(\log \phi_{0,1/k})''-4\tn_K(s)(\log\phi_{0,1/k})')u+(2(\log \phi_{0,1/k})'-2\tn_K(s))u'\\
u(z,0) = \psi_{k,0}(z)-(\log \phi_{0,1/k})'(z)\\
u(0,t)=u(\tfrac{D}{2},t) = 0.
\end{cases}
\end{equation*}
The corresponding parabolic operator (as in \cite{he-wei}) is given by
\begin{equation*}
Pu= -u_t+u''+a(z,u,u')
\end{equation*}
where the lower order term $a(z,u,u')$ is given by
\begin{equation*}
2uu'+a_1u'+a_2u-2\tn_K(s)u^2,
\end{equation*}
with
\begin{align*}
a_1 &= (2(\log \phi_{0,1/k})'-2\tn_K(s))\\
a_2 &= (2(\log \phi_{0,1/k})''-4\tn_K(s)(\log\phi_{0,1/k})').
\end{align*}
Then we have the following maximum principle
\begin{lemma}[Lemma 4.1 \cite{he-wei}]
Suppose that $u,v \in C^{2,1}(R_T)\cap C(\bar R_T)$ such that $Pu \geq Pv$ in $R_T$ and $u\leq v$ on $\mathcal{P}(R_T)$.  Assume that either $u_z$ or $v_z$ has an upper bound on $R_T$, then $u\leq v$ on $\bar{R}_T$.
\end{lemma}
Here $R_T = (0,D/2)\times (0,T]$, $\mathcal{P}(R_T)$ is the parabolic boundary, and $C^{2,1}$ means $C^2$ in the spacial variable and $C^1$ in the $t$ variable.
From here the same argument (in \S 4, \S 5 of \cite{he-wei}) follows.  Namely one applies the maximum principle to show that $\psi_k(z, t)$ is sandwiched between $(\log\phi_{0,1/k})'(z)$ and $\psi_{k,0}(z)$. To obtain the comparison for $\psi_{k,0}$, we require that the functions $\psi^L_c$ and $\psi^R_{k,c}$ are stationary solutions.  Then applying the strong maximum principle, we get for each $k>0$ the convergence of the solution $\psi_k(z,t) \to (\log \phi_{0,{1/k}})'$ as $t \to \infty$.  Letting $k\to \infty$ gives the result.

\section{Gap Estimate}
Parallel to  \cite[Theorem 4.1]{seto-wang-wei}, we have the following gap estimate.
\begin{theorem}  \label{gap-comp}
	Let $\Omega$ be a bounded convex domain with diameter $D$ in a Riemannian manifold $M^n$ with $\Ric_M \ge (n-1)K$,  $\phi_1$ a positive first eigenfunction of the Laplacian on $\Omega$ with Dirichlet boundary condition. Assume $\phi_1$ satisfies the log-concavity estimates
	\be
	\langle \nabla \log \phi_1 (y), \gamma'(\tfrac{d}{2}) \rangle - \langle \nabla\log \phi_1(x),\gamma'(-\tfrac{d}{2})\rangle  \leq -2 \tfrac{\pi}{D} \tan \left( \tfrac{\pi d}{2D} \right)  + (n-1)\tn_K (\tfrac{d}{2}),		
	\ee
	where $\gamma$ is the unit-speed length minimizing geodesic with $\gamma(-\tfrac{d}{2}) =x$, \  $\gamma(\tfrac{d}{2}) = y$,  and  $d= d(x,y)$. 
	Then we have the gap estimate
	\begin{equation}
		\lambda_2- \lambda_1 \ge 3 \frac{ \pi^2}{D^2}.  \label{gap-comp}
	\end{equation}
\end{theorem}
The proof is similar to the proof of  \cite[Theorem 4.1]{seto-wang-wei}, but we compare to the Euclidean model instead of the  curvature $K$-sphere model. 
\begin{proof}
Let $w(x) = \frac{u_2(x)}{u_1(x)}$ and $\bar{w}(s) = \frac{\bar\phi_2(s)}{\bar\phi_1(s)}$ where $u_i$ are the first and second eigenfunctions of the Laplacian on $\Omega$ with Dirichlet boundary and $\bar\phi_i$ are the first and second eigenfunctions of the Euclidean model
\begin{equation*}
\begin{cases}
\bar\phi'' +\bar\lambda\bar\phi = 0 & \text{ on } [-D/2,D/2]\\
\bar\phi(\pm D/2) = 0.
\end{cases}
\end{equation*}
In fact, $\bar{\phi}_1(s) = \cos\left(\frac{\pi}{D} s\right), \ \bar{\phi}_2(s) = \sin \left(\frac{2\pi}{D} s\right)$, $\bar{\lambda}_1 =  \frac{ \pi^2}{D^2}, \ \bar{\lambda}_2 =  \frac{ 4\pi^2}{D^2}$, $\bar w(s) = 2\sin(\tfrac{\pi}{D}s)$, and $\left(\log \bar\phi_1\right)' (s) = -\frac{\pi}{D}\tan \left( \tfrac{\pi s}{D} \right)$. 

By direct computation,
\begin{align}
\begin{split}\label{elliptic}
\nabla w &= \frac{\nabla u_2}{u_1}-w\nabla\log u_1,\\
\Delta w &= -(\lambda_2-\lambda_1)w-2\langle \nabla \log u_1,\nabla w\rangle,\\
\bar w' &= \frac{\bar\phi_2'}{\bar\phi_1}-\frac{\bar\phi_2\bar\phi_1'}{\bar\phi_1^2} = 2\tfrac{\pi}{D}\cos(\tfrac{\pi}{D}s),\\
\bar w'' &= -(\bar\lambda_2-\bar\lambda_1)\bar w - 2(\log \bar\phi_1)'\bar w' = -2\tfrac{\pi^2}{D^2}\sin(\tfrac{\pi}{D}s).
\end{split}
\end{align}

We can extend $w$ to a smooth function on $\overline{\Omega}$ with Neumann condition $\frac{\dd w}{\dd \nu} = 0$ on $\dd\Omega$ \cite{singerwongyauyau}, same for $\bar{w}$.  Let 
\begin{align*}
Q(x,y) = \frac{w(x)-w(y)}{\bar w\left( \frac{d(x,y)}{2}\right)}
\end{align*}
on $\overline{\Omega} \times \overline{\Omega} \setminus \Delta$, where $\Delta = \{(x,x)| x \in \overline{\Omega} \}$ is the diagonal.  Since
\begin{align*}
\lim_{y \to x}Q(x, y)  = 2\frac{\langle\nabla w(x), X \rangle }{\bar{w}'(0)},
\end{align*}
where $X = \gamma'(0)$ and $\gamma$ is the unique unit speed length minimizing geodesic connecting $x$ to $y$,
we can extend the function $Q$ to the unit sphere bundle $U\Omega = \{(x,X) \ | \ x \in \bar{\Omega}, \|X\|=1\}$  as
\begin{equation*}
Q(x,X) = \frac{2\langle \nabla w(x),X\rangle}{\bar{w}'(0)}.
\end{equation*}
The maximum of $Q$ then is achieved.

Case 1:  the maximum of $Q$ is achieved at $(x_0, y_0)$ with $x_0 \neq y_0$.  Denote $d_0 = d(x_0,y_0)>0$,   $m = Q(x_0,y_0)>0$ the maximum value. At $(x_0, y_0)$,  we have  $\nabla Q = 0, \ \nabla^2 Q \le 0$.
 The Neumann condition $\frac{\dd w}{\dd \nu} = 0$ and strict convexity of $\Omega$ forces that both $x_0$ and $y_0$ must be in $\Omega$.

Let $\gamma$ be the unit-speed length minimizing geodesic such that $\gamma(-\tfrac{d_0}{2})=x_0$ and $\gamma(\tfrac{d_0}{2})=y_0$.  Let $e_n := \gamma'$ and extend to an orthonormal basis $\{e_i\}$ by parallel translation along $\gamma$.  Denote $E_i = e_i \oplus e_i$ for $i = 1, \ldots,n$; $E_n = e_n \oplus (-e_n)$.

For $E \in T_xM\oplus T_yM$,
\begin{equation}\label{Q'}
\nabla_E Q = \frac{\nabla_E w(x) - \nabla_Ew(y)}{\bar w} - \frac{(w(x)-w(y))}{\bar w^2}(\nabla_E\bar w),
\end{equation}
and
\begin{equation}\label{Q''1}
\nabla_{E,E}^2 Q = \frac{\nabla_{E,E}^2w(x)-\nabla_{E,E}^2w(y)}{\bar w}-\frac{2}{\bar w}(\nabla_E Q)(\nabla_E\bar w)-\frac{Q}{\bar w}\nabla_{E,E}^2\bar w.
\end{equation}
Hence at $(x_0,y_0)$, 
\begin{align*}
0 &= \frac{\nabla_E w(x_0)-\nabla_Ew(y_0)}{\bar w} - \frac{m}{\bar w}(\nabla_E\bar w),\\
0 &\geq \frac{\nabla^2_{E,E}w(x_0) - \nabla^2_{E,E}w(y_0)}{\bar w}-\frac{m}{\bar w}\nabla^2_{E,E}\bar w.
\end{align*}
We apply these to various directions.  From $\nabla_{0\oplus e_i}Q = \nabla_{e_i\oplus 0}Q = 0$ so that
\begin{align*}
\nabla_{e_i}w(y_0) = \nabla_{e_i}w(x_0) = 0
\end{align*}
for $i= 1, \ldots, n-1$ and
\begin{equation*}
\nabla_{e_n}w(y_0) = \nabla_{e_n}w(x_0) = -\frac{m}{2}\bar w'(\tfrac{d_0}{2}).
\end{equation*}
so that the full gradient is given by
\begin{equation*}
\nabla w(y_0) = \nabla w(x_0) = -\frac{m}{2}\bar w'(\tfrac{d_0}{2})e_n.
\end{equation*}
Summing over the second order inequalities, we get
\begin{align*}
0 &\geq \frac{\Delta w(x_0)-\Delta w(y_0)}{\bar w} - \frac{m}{\bar w}\sum_{i=1}^n\nabla^2_{E_i,E_i}\bar w (\tfrac{d_0}{2}).
\end{align*}
Since $\bar w' \geq 0$, by the ``Two Point Laplacian Comparison" (see e.g. \cite[(4.5)]{seto-wang-wei}) we have $\sum_{i=1}^{n-1}\nabla^2_{E_i,E_i}\bar w (\tfrac{d_0}{2}) \leq -(n-1)\tn_K (\tfrac{d_0}{2})\,  \bar w'(\tfrac{d_0}{2})$.  Plugging this in, and using \eqref{elliptic}, we get
\begin{align*}
0 &\geq -(\lambda_2-\lambda_1)m+2\frac{\langle \nabla\log u_1(y_0),\nabla w(y_0)\rangle - \langle \nabla \log u_1(x_0),\nabla w(x_0)\rangle}{\bar w} + (n-1)\frac{m}{\bar w}\tn_K\bar{w}' -\frac{m}{\bar w}\bar w'' \\
&=-(\lambda_2-\lambda_1)m+(\bar\lambda_2-\bar\lambda_1)m\\
&\hspace{0.2 in}+2m(\log\bar\phi_1)'\frac{\bar w'}{\bar w}-m\bar w'\frac{\langle \nabla\log u_1(y_0),e_n\rangle - \langle \nabla \log u_1(x_0),e_n\rangle}{\bar w} + (n-1)\frac{m}{\bar w}\tn_K\bar{w}' \\
&\geq -(\lambda_2-\lambda_1)m+(\bar\lambda_2-\bar\lambda_1)m,
\end{align*}
which is (\ref{gap-comp}).

Case 2:  the maximum of $Q$ is attained at some $(x_0, X_0) \in U\Omega$.  
  By Cauchy-Schwarz inequality, the corresponding maximal direction is $X_0 = \frac{\nabla w}{\|\nabla w\|}$ so that the maximum value is $m = \frac{D}{\pi}\|\nabla w\|$.  Furthermore, $\|\nabla w(x_0)\| \geq \|\nabla w(x)\|$ for any $x \in \bar{\Omega}$.  Suppose $x_0 \in \dd\Omega$, then by (strict) convexity,
\begin{align*}
\nabla_n \|\nabla w\|^2 |_{x_0} = -\SecFun(\nabla w,\nabla w)|_{x_0} < 0
\end{align*}
hence the maximum must occur in the interior.  Now let $e_n := \frac{\nabla w}{\|\nabla w\|}$ and complete to an orthonormal frame $\{e_i\}$ at $x_0$.  We further parallel translate to a neighborhood of $x_0$.  In such a frame we have
\begin{equation*}
\nabla_n w = \langle \nabla w, e_n\rangle = \|\nabla w\|
\end{equation*}
and
\begin{equation*}
\nabla_i w = \langle \nabla w, e_i\rangle = 0, \quad i=1,\ldots, n-1
\end{equation*}
At the maximal point $x_0$, we have the first derivative vanishing
\begin{equation*}
0=\nabla\|\nabla w\|^2 = 2\langle \nabla\nabla w, \nabla w\rangle = 2\|\nabla w\|\nabla_n \nabla w,
\end{equation*}
and the second derivative non-positive
\begin{align*}
0 &\geq \nabla_k\nabla_k \|\nabla w\|^2 \\
&=2\left(\langle \nabla_k\nabla_k\nabla w,\nabla w\rangle + \|\nabla_k\nabla w\|^2\right)\\
&\geq 2\langle \nabla_k\nabla_k\nabla w,\nabla w\rangle \\
&=2\|\nabla w\| \langle \nabla_k\nabla_k\nabla w,E_n\rangle.
\end{align*}
In short
\begin{equation}\label{gapestsecondvar}
0 \geq \langle \nabla_k\nabla_k \nabla w, e_n\rangle, \quad k=1,\ldots n-1.
\end{equation}
Now let
\begin{align*}
&x(s) := \exp_{x_0}(se_n)\\
&y(s) := \exp_{x_0}(-se_n)\\
&g(s) := Q(x(s),y(s)).
\end{align*}
By construction, since the variations are approaching $x_0$ in the $e_n$ direction, we have
\begin{equation*}
m = Q(x_0,e_n(x_0))= g(0) \geq g(s), \quad \text{ for all } s\in (-\eps,\eps).
\end{equation*}
and so $\lim_{s\to 0} g'(s) = 0$ and $\lim_{s\to 0} g''(s) \leq 0$.  By (\ref{Q'}), (\ref{Q''1})
\begin{align*}
g'(s) 
&=\frac{\langle \nabla w,x'(s)\rangle - \langle \nabla w,y'(s)\rangle}{\bar{w}(s)} - \frac{g(s)}{\bar{w}(s)} \bar{w}', \\
g''(s) &= \frac{\langle \nabla_s\nabla w(x(s)),x'(s)\rangle + \langle \nabla w(x(s)),x''(s)\rangle - \langle \nabla_s\nabla w(y),y'(s)\rangle - \langle \nabla w,y''(s)\rangle}{\bar{w}}\\
&-\frac{\langle \nabla w,x'(s)\rangle- \nabla w,y'(s)\rangle}{\bar{w}}\frac{\bar{w}'}{\bar{w}}-g'(s)\frac{\bar{w}'}{\bar{w}}-g(s)\left(\frac{\bar{w}''}{\bar{w}}-\left(\frac{\bar{w}'}{\bar{w}}\right)^2\right).
\end{align*}
Using $\bar{w}'' = -\frac{\pi^2}{D^2}\bar{w}$ and
\begin{align*}
x''(s) &= \frac{d}{ds} x'(s) \\
&=\nabla_{x'(s)}x'(s) = 0,
\end{align*}
and similarly for $y''(s)$, when $s\to 0$ we have
\begin{align*}
0 &\geq 2\frac{\langle \nabla_n\nabla_n \nabla w,e_n\rangle}{\bar{w}'(0)} + m\frac{\pi^2}{D^2}
\end{align*}
Combining this with \eqref{gapestsecondvar}, we have
\begin{equation*}
0 \geq 2\frac{\langle \Delta(\nabla w),e_n \rangle}{\bar{w}'(0)} + m\frac{\pi^2}{D^2}.
\end{equation*}
By Bochner formula,
\begin{equation*}
0 \geq 2\frac{\langle \nabla (\Delta w),e_n \rangle +\Ric(\nabla w, e_n)}{\bar{w}'(0)} + m\frac{\pi^2}{D^2}.
\end{equation*}
Inserting in \eqref{elliptic}, we have
\begin{align*}
0 &\geq 2\frac{\langle \nabla (-(\lambda_2 - \lambda_1)w - 2\langle\nabla\log u_1,\nabla w\rangle ), e_n \rangle +\Ric(\nabla w, e_n)}{\bar{w}'(0)} + m\frac{\pi^2}{D^2} \\
& = (-2(\lambda_2-\lambda_1)- 4\langle \nabla_n \nabla\log u_1, e_n\rangle + 2\Ric(e_n,e_n))\frac{\|\nabla w\|}{\bar{w}'(0)} + m\frac{\pi^2}{D^2}
\end{align*}
From the log-concavity
\begin{equation*}
\frac{\langle \nabla \log u_1 (y), \gamma'(\tfrac{d}{2}) \rangle - \langle \nabla\log u_1(x),\gamma'(-\tfrac{d}{2})\rangle}{d(x,y)} \leq -2\frac{\pi}{D}\frac{\tan\left(\frac{\pi d(x,y)}{2D} \right)}{d(x,y)} + (n-1)\frac{\tn_K(\tfrac{d(x,y)}{2})}{d(x,y)},
\end{equation*}
and letting $d(x,y) \to 0$ we have $-\nabla^2\log u_1 \geq \frac{\pi^2}{D^2}-\frac{(n-1)K}{2}$.  Using the fact that $\bar{w}'(0) = \frac{\pi}{D}$ and $m= \frac{D}{\pi}\|\nabla w\|$,
\begin{equation*}
(\lambda_2 - \lambda_1) \geq 3\frac{\pi^2}{D^2} 
\end{equation*}
\end{proof}

\section{Eigenvalue Asymptotics of the Sphere Model}
First we recall the derivation of the one-dimensional model used in \cite{seto-wang-wei}.  Let $\mathbb M^n_K$ be the $n$-dimensional simply connected manifold with constant sectional curvature $K$.   Given  a totally geodesic hypersurface  $\Sigma \subset \mathbb M^n_K$, let $s$ be the (signed) distance to $\Sigma$. 
The metric of $\mathbb M^n_K$ (near $\Sigma$) can be written as  $
g = ds^2 + \cs^2_K (s) g_{\Sigma}. $
This is different from the usual polar coordinate model,  and $s$ can be negative here. 

The Laplace operator is $$ \Delta = \tfrac{\partial^2}{\partial s^2 } + (n-1) \tfrac{\cs_K'(s)}{\cs_K(s)} \tfrac{\partial}{\partial s} + \tfrac{1}{\cs_K^2(s)} \Delta_{\Sigma}. $$
The ``one-dimensional" model of the eigenvalue equation $\Delta \phi = -\lambda \phi$ (when $\phi$ only depends on $s$) is
\begin{equation}\label{onedimmodel}
\phi''-(n-1)\tn_K(s)\phi' + \lambda \phi = 0.
\end{equation}
With the change of variable $\phi (s) = \cs_K^{-\frac{n-1}{2}}(s) \varphi(s)$,  we obtain the  Schr\"odinger normal form of \eqref{onedimmodel},
\begin{equation}\label{schrodingernormal}
\varphi''(s) - \tfrac{(n-1)K}{4} \left( \tfrac{n-3}{\cs_K^2(s)}  - (n-1) \right)  \varphi=  - \lambda \, \varphi.
\end{equation}
Hence the Dirichlet eigenvalues of \eqref{onedimmodel} are exactly  the same as the Dirichlet eigenvalues of \eqref{schrodingernormal}. 
Denote $\bar{\lambda}_1 (n,D,K), \ \bar{\lambda}_2 (n,D,K)$ their first and second Dirichlet eigenvalues on $[-\tfrac{D}{2},  \tfrac{D}{2}]$.  When $n=1,3$ or $K=0$, one can find the eigenvalues and eigenfunctions explicitly and the gap  $ \bar{\lambda}_2 (n,D,K) -\bar{\lambda}_1 (n,D,K) = 3 \frac{ \pi^2}{D^2}$.  In general one can not find the eigenvalues explicitly. 
When $K>0$,  as $(\cs_K^{-2}(s))'' \ge 0$, $ \bar{\lambda}_2 (n,D,K) -\bar{\lambda}_1 (n,D,K) > 3 \frac{ \pi^2}{D^2}$ when  $n > 3$,  but $<3 \frac{ \pi^2}{D^2}$ when  $n =2$ \cite{ashbaughbenguria89}.

First we note some easy bounds on these model eigenvalues.

\begin{proposition}
For $K>0$, we have 
\begin{equation*}
\bar{\lambda}_1 \leq \frac{\pi^2}{D^2}-\frac{(n-1)^2K}{4}+\frac{(n-1)(n-3)K}{D}\int_0^{D/2}\sec^2(\sqrt{K}x)\cos^2(\tfrac{\pi}{D}x),
\end{equation*}
while if $K>0$ and $n \ge 3$, one has
\begin{equation*}
\bar{\lambda}_1 \ge  \frac{\pi^2}{D^2} -\frac{(n-1)K}{2}.
\end{equation*}
Similarly for $\bar{\lambda}_2$, we have
 \begin{align*}
 \bar{\lambda}_2 &\leq \frac{4\pi^2}{D^2}-\frac{(n-1)^2K}{4}+\frac{(n-1)(n-3)K}{D}\int_0^{D/2}\sec^2(\sqrt{K}x)\sin^2(\tfrac{2\pi}{D}x),
 \end{align*}
 whereas if $n \ge 3$, 
 \begin{equation*}
\bar{\lambda}_2 \ge  \frac{4\pi^2}{D^2}-\frac{(n-1)K}{2}. 	
\end{equation*}
For $n=2$, the upper bounds can be made more explicit, see (\ref{lambda_1n=2}), (\ref{lambda_2n=2}).
\end{proposition}

\begin{proof}
For $K>0$, as
	\begin{align*}
\bar{\lambda}_1 = \inf_{f \in C_0([-D/2,D/2])}\frac{\int_{-D/2}^{D/2}(f')^2}{\int_{-D/2}^{D/2}f^2}+\frac{(n-1)(n-3)K}{4}\frac{\int_{-D/2}^{D/2}\sec^2(\sqrt{K}x)f^2}{\int_{-D/2}^{D/2}f^2}-\frac{(n-1)^2K}{4},
	\end{align*}
	and $\sec^2(\sqrt{K}x) \ge 1$, we have, for $n \ge 3$, 
	\begin{equation}
\bar{\lambda}_1 \ge  \frac{\pi^2}{D^2}-\frac{(n-1)K}{2}.  \label{lambda_1-lb}	
	\end{equation}
For an upper bound, let $f = \cos(\tfrac{\pi}{D}x)$, we have  
\begin{align*}
\bar{\lambda}_1 &\leq \frac{\pi^2}{D^2}-\frac{(n-1)^2K}{4}+\frac{(n-1)(n-3)K}{D}\int_0^{D/2}\sec^2(\sqrt{K}x)\cos^2(\tfrac{\pi}{D}x).
\end{align*}
When $n=2$, we can get the following explicit upper bound by using $\sec^2(t) \geq 1 + t^2 + \frac{2t^4}{3}$, 
\begin{align}
\bar{\lambda}_1 \leq \frac{\pi^2}{D^2}-\frac{K}{2} -\frac{(\pi^2-6)D^2K^2}{48\pi^2}-\frac{(120-20\pi^2+\pi^4)D^4K^3}{480\pi^4}-\frac{17(\pi^6-42\pi^4+840\pi^2-5040)D^6K^4}{80640\pi^6}.  \label{lambda_1n=2}
\end{align}
Similarly for $\bar{\lambda}_2$, we have 
	\begin{equation*}
\bar{\lambda}_2 \ge  \frac{4\pi^2}{D^2}-\frac{(n-1)K}{2}. 	
\end{equation*}
For an upper bound we can use $f=\sin(\tfrac{2\pi}{D}x)$ as a test function since the first eigenfunction of the model is even, and we get
 \begin{align*}
 \bar{\lambda}_2 &\leq \frac{4\pi^2}{D^2}-\frac{(n-1)^2K}{4}+\frac{(n-1)(n-3)K}{D}\int_0^{D/2}\sec^2(\sqrt{K}x)\sin^2(\tfrac{2\pi}{D}x).
 \end{align*}
 When $n=2$,
\begin{align} 
\bar\lambda_2 &\leq \frac{4\pi^2}{D^2}-\frac{K}{2}- \frac{(\pi^2-\tfrac 32 )D^2K^2}{48\pi^2}-\frac{(\frac{15}{2}-5\pi^2+\pi^4)D^4K^3}{480\pi^4} - \frac{17(4\pi^6-42\pi^4+210\pi^2-315)D^6K^4}{322560}.
\label{lambda_2n=2}
\end{align}
\end{proof}
Obtaining explicit lower bounds for $\bar{\lambda}_1$ and  $\bar{\lambda}_2$ up to second order of $K$ is surprisingly hard.  Here we compute the asymptotic expansion of the eigenvalues $\bar{\lambda}_1 (n,D,K), \ \bar{\lambda}_2 (n,D,K)$ in terms of powers of the curvature $K$, proving  Proposition~\ref{eigen-asym} which we state here again for convenience. 
 
\begin{proposition}
For $K\in \mathbb{R}$, let $\kappa = KD^2$.  Then
\begin{equation*}
D^2\bar\lambda_1 = \pi^2-\frac{(n-1)}{2}\kappa+\frac{(n-1)(n-3)}{48\pi^2}(\pi^2-6)\kappa^2+\frac{(n-1)(n-3)}{480\pi^4}(\pi^4-20\pi^2+120)\kappa^3 +O(\kappa^4).
\end{equation*}
and
\begin{equation*}
D^2\bar\lambda_2 = 4\pi^2-\frac{(n-1)}{2}\kappa+\frac{(n-1)(n-3)}{48\pi^2}\left(\pi^2-\frac{3}{2}\right)\kappa^2+\frac{(n-1)(n-3)}{480\pi^4}\left(\pi^4-5\pi^2+\frac{15}{2}\right)\kappa^3 +O(\kappa^4).
\end{equation*}
Hence \[
D^2(\bar{\lambda}_2 -\bar{\lambda}_1 ) =   3  \pi^2 + \frac{3(n-1)(n-3)}{32 \pi^2}  \kappa^2 + \frac{(n-1)(n-3)}{480\pi^4}\left(15\pi^2-\frac{225}{2}\right)\kappa^3 +O(\kappa^4).	\]	
\end{proposition}

\begin{proof}
We shift the eigenvalue by $\frac{(n-1)^2}{4}K$ and perturb about $K=0$.  First set $D=\pi$.  Then the $K=0$ solution is given by $ \cos(x)$.  Set
\begin{align*}
y &= \cos(x) + Ky_{1,1}+K^2y_{1,2} +K^3y_{1,3},\\
\tilde\lambda_1 &= 1 + K\lambda_{1,K}+K^2\lambda_{1,K^2} +K^3\lambda_{1,K^3},
\end{align*}
where $\tilde\lambda_1$ is the shifted first eigenvalue.  Expanding $\sec^2(\sqrt{K}x) = 1 +Kx^2 + \frac{2}{3}K^2x^4 +\cdots $ and plugging in our expansion solutions, the first order equation in $K$ is given by
\begin{equation*}
y_{1,1}'' + y_{1,1} = \left(\frac{(n-1)(n-3)}{4}-\lambda_{1,K}\right)\cos(x).
\end{equation*}
Using the fact that the first eigenfunction is even about $x=0$, the particular solution is of the form $Ax\sin(x)$.  Plugging this in and using the Dirichlet boundary condition leads to
\begin{align*}
\lambda_{1,K} = \frac{(n-1)(n-3)}{4}
\end{align*}
Using the expansion again and plugging in for $\lambda_{1,K}$, the $K^2$ order equation is
\begin{align*}
y_{1,2}''+y_{1,2} -\left(\frac{(n-1)(n-3)}{4}x^2-\lambda_{1,K^2}\right)\cos(x) = 0.
\end{align*}
Using the fact that the solution is even, the particular solution is of the form $y_p = Ax^2\cos(x) + (Bx^3+Cx)\sin(x)$.  Plugging this in and using the Dirichlet condition again gives us
\begin{equation*}
\lambda_{1,K^2} = \frac{(n-1)(n-3)}{24}\left(\frac{\pi^2}{2}-3\right).
\end{equation*}
The $K^3$ equation is 
\begin{align*}
y_{1,3}''+y_{1,3} = \frac{(n-1)(n-3)}{6}x^4\cos(x)- \lambda_{1,K^3}\cos(x).
\end{align*}
Similar computations give
\begin{equation*}
\lambda_{1,K^3} = \left(\pi^4-20\pi^2+120\right)\frac{(n-1)(n-3)}{480}.
\end{equation*}
Combining these and shifting by $\frac{(n-1)^2K}{4}$, we get
\begin{equation*}
\bar{\lambda}_1 = 1-\frac{(n-1)}{2}K+\frac{(n-1)(n-3)}{48}(\pi^2 - 6)K^2 + \left(\pi^4-20\pi^2+120\right)\frac{(n-1)(n-3)}{480}K^3 +O(K^4).
\end{equation*}
By rescaling, we obtain
\begin{equation*}
\bar\lambda_1 = \frac{\pi^2}{D^2}-\frac{(n-1)}{2}K+\frac{(n-1)(n-3)}{48}(\pi^2 - 6)\frac{D^2}{\pi^2}K^2 +\frac{(n-1)(n-3)}{480}\frac{D^4}{\pi^4}K^3(\pi^4-20\pi^2+120) +O(K^4).
\end{equation*}
To compute the asymptotics of the second eigenvalue, we repeat the steps above and instead we use the second eigenfunction solution for the $K=0$ case so that
\begin{align*}
y &= \sin(2x) + Ky_{2,1}+K^2y_{2,2} +K^3y_{2,3},\\
\tilde\lambda_2 &= 4+K\lambda_{2,K}+K^2\lambda_{2,K^2}+K^3\lambda_{2,K^3},
\end{align*}
where again, $\tilde\lambda_2$ is the shifted eigenvalue. 
\end{proof}

\subsection{Higher Order terms} \label{high-order}
\subsubsection{Fourth order term}
Beginning with the fourth order term, the sign of the coefficient changes for some $n >3$ instead of at $n=3$.   We compute for $D=\pi$.  Expanding out the equation and collecting the $K^4$ terms, we have
\begin{align*}
y_{1,4}''+y_{1,4}+\lambda_{1,K^4}\cos(x)=\left(\frac{(n-1)(n-3)}{4}x^2-\lambda_{1,K^2}\right)y_{1,2}+\frac{17(n-1)(n-3)}{180}x^6\cos(x).
\end{align*}
Multiplying by $\cos(x)$, integrating from $-\frac{\pi}{2}$ to $\frac{\pi}{2}$ and using the second order equation, we get
\begin{align*}
\lambda_{1,K^4} &= \frac{2}{\pi}\left(\frac{17(n-1)(n-3)}{180}\int_{-\pi/2}^{\pi/2}x^6\cos^2(x)dx-\int_{-\pi/2}^{\pi/2}(y'_{1,2})^2dx +\int_{-\pi/2}^{\pi/2}y_{1,2}^2dx\right).
\end{align*}
From the computation of the second order term, we have the second order of the first eigenfunction
\begin{align*}
y_{1,2} = \frac{(n-1)(n-3)}{24}\left(x^3\sin(x)+\frac{3}{2}x^2\cos(x)-\frac{\pi^2}{4}x\sin(x)\right).
\end{align*}
Using this, we get
\begin{align*}
\lambda_{1,K^4} = \frac{(n-1)^2(n-3)^2}{24^2}\left(\frac{\pi^4-75\pi^2+630}{20}\right)+\frac{(n-1)(n-3)}{24}\left(\frac{17(\pi^6-42\pi^4+840\pi^2-5040)}{3360}\right).
\end{align*}
Note that
\begin{align*}
\frac{\pi^4-75\pi^2+630}{20} \approx -0.64 \end{align*}
\begin{align*}\frac{17(\pi^6-42\pi^4+840\pi^2-5040)}{3360}\approx 0.61.
\end{align*}
Similar computations yield
\begin{align*}
\lambda_{2,K^4} =  \frac{(n-1)^2(n-3)^2}{24^2}\left(\frac{8\pi^4-150\pi^2+315}{640}\right)+\frac{(n-1)(n-3)}{24}\left(\frac{17(4\pi^6-42\pi^4+210\pi^2-315)}{13440}\right).
\end{align*}
Note that
\begin{align*}
\frac{8\pi^4-150\pi^2+315}{640} \approx  -0.603
\end{align*}
\begin{align*}
\frac{17(4\pi^6-42\pi^4+210\pi^2-315)}{13440} \approx 1.912
\end{align*}
and the gap is
\begin{align*}
\lambda_{2,K^4}-\lambda_{1,K^4} = \frac{(n-1)^2(n-3)^2}{24^2}\left(\frac{3(750\pi^2-8\pi^4-6615)}{640}\right) + \frac{(n-1)(n-3)}{24}\left(\frac{51(2\pi^4-50\pi^2+315)}{640}\right).
\end{align*}
Note that
\begin{align*}
\left(\frac{51(2\pi^4-50\pi^2+315)}{640}\right) \approx 1.301
\end{align*}
and
\begin{align*}
\frac{3(750\pi^2-8\pi^4-6615)}{640} \approx 0.037.
\end{align*}

\subsubsection{Fifth order term}
To compute the fifth order term, we need the third order eigenfunctions.
\begin{align*}
y_{1,3} = \frac{(n-1)(n-3)}{24}\left((x^4-3x^2)\cos(x) +\left(\frac{2}{5}x^5-2x^3-\left(\frac{\pi^4}{40}-\frac{\pi^2}{2}\right)x\right)\sin(x)\right)
\end{align*}
and
\begin{align*}
y_{2,3} &= -\frac{(n-1)(n-3)}{120}\left(\left(x^5-\frac{5}{4}x^3+\left(\frac{5\pi^2}{16}-\frac{\pi^4}{16}\right)x\right)\cos(2x)-\left(\frac{5}{4}x^4-\frac{15}{16}x^2\right)\sin(2x)\right).
\end{align*}
Then
\begin{align*}
\int_{-\pi/2}^{\pi/2}y_{1,2}y_{1,3} &= \frac{(n-1)^2(n-3)^2}{24^2}\frac{1}{160}\left(-15570\pi+2220\pi^3-67\pi^5+\frac{13\pi^7}{168}+\frac{4\pi^9}{315}\right)\\
& \approx \frac{(n-1)^2(n-3)^2}{24^2}0.1766
\end{align*}
and
\begin{align*}
\int_{-\pi/2}^{\pi/2} (y_{1,2}')(y_{1,3}') &= \frac{(n-1)^2(n-3)^2}{24^2} \frac{1}{160}\left(1710\pi - 300\pi^3+17\pi^5-\frac{83\pi^7}{168}+\frac{4\pi^9}{315}\right)\\
& \approx\frac{(n-1)^2(n-3)^2}{24^2} 0.993 
\end{align*}
and
\begin{align*}
\frac{62}{315}\int_{-\pi/2}^{\pi/2}x^8\cos^2(x) = \frac{62}{315}\frac{\pi(362880-60480\pi^2+3024\pi^4-72\pi^6+\pi^8)}{4608} \approx 0.10734.
\end{align*}
Combining these together, we have
\begin{align*}
\lambda_{1,5} &= \frac{2}{\pi}\left(2\int y_{1,2}y_{1,3} - 2\int (y_{1,2}')(y_{1,3}')+\frac{(n-1)(n-3)}{4}\frac{62}{315}\int x^8\cos^2(x)  \right)\\
&= \frac{(n-1)^2(n-3)^2}{24^2}\frac{(-30240+4410\pi^2-147\pi^4+\pi^6)}{70} +\frac{(n-1)(n-3)}{2\pi}\frac{62}{315}\int x^8\cos^2(x)\\
&\approx \frac{(n-1)^2(n-3)^2}{24^2}(-1.039)  + \frac{(n-1)(n-3)}{2\pi}(0.10734).
\end{align*}

For the second eigenvalue,
\begin{align*}
\int_{-\pi/2}^{\pi/2}y_{2,2}y_{2,3} &= \frac{(n-1)^2(n-3)^2}{2\times 48^2}\frac{1}{80}\left(-\frac{7785\pi}{128}+\frac{555\pi^3}{16}-\frac{67\pi^5}{16}+\frac{13\pi^7}{672}+\frac{4\pi^9}{315}\right) \\
&\approx \frac{(n-1)^2(n-3)^2}{ 48^2}0.249
\end{align*}
and
\begin{align*}
\int_{-\pi/2}^{\pi/2}(y_{2,2}')(y_{2,3}') &= \frac{(n-1)^2(n-3)^2}{48^2}\frac{1}{160}\left(\frac{855\pi}{32}-\frac{75\pi^3}{4}+\frac{17\pi^5}{4}-\frac{83\pi^7}{168}+\frac{16\pi^9}{315}\right)\\
&\frac{(n-1)^2(n-3)^2}{48^2}\approx 5.157
\end{align*}
and
\begin{align*}
\frac{(n-1)(n-3)}{4}\frac{62}{315}\int_{-\pi/2}^{\pi/2} x^8\sin^2(2x) &= \frac{(n-1)(n-3)}{4}\frac{31\pi(2835-1890\pi^2+378\pi^4-36\pi^6+2\pi^8)}{1451520}\\
&\approx \frac{(n-1)(n-3)}{4}0.36024.
\end{align*}
Combining these together,
\begin{align*}
\lambda_{2,5} &= \frac{2}{\pi}\left(2\int y_{2,2}y_{2,3} - 2\int(y'_{2,2})(y'_{2,3}) + \frac{(n-1)(n-3)}{4}\frac{62}{315}\int_{-\pi/2}^{\pi/2}x^8\sin^2(2x)\right)\\
&=\frac{(n-1)^2(n-3)^2}{48^2}\left( -\frac{2241}{1024}+\frac{171\pi^2}{128}-\frac{27\pi^4}{128}+\frac{23\pi^6}{1792}-\frac{\pi^8}{1050} \right) + \frac{(n-1)(n-3)}{2\pi}\frac{62}{315}\int_{-\pi/2}^{\pi/2}x^8\sin^2(2x)\\
&\approx \frac{(n-1)^2(n-3)^2}{24^2}(-1.561)+ \frac{(n-1)(n-3)}{2\pi}(0.35024)
\end{align*}
so that
\begin{align*}
\lambda_{2,5}-\lambda_{1,5} \approx \frac{(n-1)^2(n-3)^2}{24^2}(-0.522)+ \frac{(n-1)(n-3)}{2\pi}(0.2429).
\end{align*}
\begin{remark}
Here we see that the sign of the coefficient of the gap changes for some large $n$.
\end{remark}

\subsection{Formula for general order}
In general,
\begin{align*}
\sum_{n=0}K^ny_{1,n}''-\frac{(n-1)(n-3)}{4}K\left(\sum_{n=0}a_nK^nx^{2n}\right)\left( \sum_{n=0}K^ny_{1,n} \right)=-\left(\sum_{n=0}K^n\lambda_n \right)\left( \sum_{n=0}K^ny_{1,n} \right),
\end{align*}
where $a_i$ are the coefficients in the series expansion of $\sec^2(x)$ and $y_{1,j}$ are the $j$-th order functions of the first eigenfunction.  Grouping the $K^m$ term, the equation becomes
\begin{align*}
y_{1,m}''+y_{1,m}+\lambda_{1,m}\cos(x)=\frac{(n-1)(n-3)}{4}\sum_{i+j=m-1}a_{i}x^{2i}y_{1,j}-\sum_{\substack{i+j=m \\ i,j< m}}\lambda_iy_{1,j}.
\end{align*}
Multiplying by $\cos(x)$ and integrating to isolate $\lambda_m$, noting that multiplying by the zero-th order eigenfunction and integrating will zero out the $m$-th order eigenfunctions.  We have
\begin{align*}
\lambda_{1,m}\int_{-\pi/2}^{\pi/2}\cos^2(x) = \frac{(n-1)(n-3)}{4}\sum_{i+j=m-1}a_i\int_{-\pi/2}^{\pi/2} x^{2i}y_{1,j}\cos(x)-\sum_{\substack{i+j=m \\ i,j <m}}\lambda_{1,i}\int_{-\pi/2}^{\pi/2} y_{1,j}\cos(x)
\end{align*}
Collecting the $j$-th terms, we get
\begin{align*}
\lambda_{1,m} = \frac{2}{\pi}\sum_{j=1}^{m-1}\int_{-\pi/2}^{\pi/2} \left(\frac{(n-1)(n-3)}{4}a_{m-j-1}x^{2(m-j-1)}-\lambda_{1,m-j}\right)\cos(x)y_{1,j}.
\end{align*}

Finally, we end with some remark about the modulus of convexity model used here compared to that used in the sphere model.
\begin{remark} \label{two-mod} There is no direct comparison between the modulus of the two models. 

Let 
\begin{align*}
f(s)=-\frac{\pi}{D}\tan(\tfrac{\pi}{D}s) + \frac{(n-1)}{2}\tn_K(s).
\end{align*}
and
\begin{align*}
\psi(x) = (\log\phi(x))',
\end{align*}
where $\phi$ satisfies
\begin{align*}
\phi''(x) -(n-1)\tn_K(x)\phi'(x) +\bar\lambda_1\phi(x) = 0.
\end{align*}
Then
\begin{align*}
\psi'(x)&= - \psi^2+(n-1)\tn_K(x)\psi(x)-\bar\lambda_1,
\end{align*}
and
\begin{align*}
f' &=-f^2+(n-1)\tn_K(s)f-\frac{\pi^2}{D^2}+\frac{(n-1)K}{2}-\frac{(n-1)(n-3)}{4}\tn_K^2(s).
\end{align*}
When $n=3$, we have $\psi =f$. In general, from (\ref{lambda_1-lb}) and (\ref{lambda_1n=2}), $\bar\lambda_1 > \frac{\pi^2}{D^2} -\frac{(n-1)K}{2}$ when $n> 3$, and $\bar\lambda_1 < \frac{\pi^2}{D^2} -\frac{K}{2}$ when $n=2$,  however the sign of the remaining term is in the opposite direction. Hence there is no direct comparison between $f$ and $\psi$. The asymptotic expansion is given by the following computation
\begin{align*}
\phi(x) &=\cs_K(x)^{-\frac{n-1}{2}}\cos(\tfrac{\pi}{D}x)\left(1+A_nK^2\left(\frac{\pi}{D}x^3\tan(\tfrac{\pi}{D}x)+\frac{3}{2}x^2-\frac{D\pi}{4}x\tan(\tfrac{\pi}{D}x)\right) + O(K^3)\right),
\end{align*}
where $A_n = \frac{(n-1)(n-3)}{24}$. Using $\log(1+x) = x-\frac{x^2}{2}+\frac{x^3}{3}+O(x^4)$,
\begin{align*}
\log(\phi) &= \log(\cos(\tfrac{\pi}{D}x))-\frac{(n-1)}{2}\log\cs_K(x) +K^2\frac{(n-1)(n-3)}{24}\left(\frac{\pi}{D}x^3\tan(\tfrac{\pi}{D}x)+\frac{3}{2}x^2-\frac{D\pi}{4}x\tan(\tfrac{\pi}{D}x)\right)\\
&\hspace{0.2 in} +O(K^4)
\end{align*}
Hence the modulus for the sphere model is asymptotically
\begin{align*}
\psi(x) &= (\log\phi(x))'\\
&=-\frac{\pi}{D}\tan(\tfrac{\pi}{D}x)+\frac{(n-1)}{2}\tn_K(x)\\
&\hspace{0.2 in}+\frac{(n-1)(n-3)}{24}K^2\left(\tfrac{\pi^2}{D^2}x^3\sec^2(\tfrac{\pi}{D}x)+\tfrac{3\pi}{D}x^2\tan(\tfrac{\pi}{D}x)+3x-\tfrac{\pi^2}{4}x\sec^2(\tfrac{\pi}{D}x)-\tfrac{D\pi}{4}\tan(\tfrac{\pi}{D}x)  \right)+O(K^3)
\end{align*}
At $x=0$, the function part of the $K^2$ term is $0$ and decreasing.  Hence for small values of $x$, the term is negative and depending on the sign of $A_n$, gives a better modulus estimate than the Euclidean model.  However the term goes to infinity as it approaches $D/2$. Compare this to the expansion of the Euclidean model
\begin{align*}
f(x) = -\frac{\pi}{D}\tan(\tfrac{\pi}{D}x) +\frac{(n-1)}{2}\tn_K(x).
\end{align*}
\end{remark}

\end{document}